\documentclass[10pt]{article}%
\usepackage{amsmath}
\usepackage{amssymb}
\usepackage{amsbsy}
\usepackage{amsthm}
\usepackage{epsfig}
\usepackage{wrapfig}
\usepackage{xcolor}
\usepackage{graphicx}
\usepackage{colordvi}
\usepackage{color}
\usepackage{graphics}
\usepackage[left=1in, top=1in, right=1in, bottom=1in]{geometry}
\usepackage{amsfonts}
\usepackage{verbatim}
\usepackage{tabularx}
\usepackage{subfigure}
\usepackage{wrapfig}
\usepackage{floatrow}
\usepackage{esvect}  
\usepackage[utf8]{inputenc}  

\numberwithin{equation}{section}
\newcommand{\norm}[1]{\left\Vert#1\right\Vert}
\newtheorem{theorem}{Theorem}[section]

\newtheorem{lemma}{Lemma}[section]

\newtheorem{definition}{Definition}[section]

\newtheorem{remark}{Remark}[section]

\newtheorem{algorithm}{Algorithm}[section]

\newcommand{\bff}{{  f}}

\def\norm#1{\|#1\|}

\def \b0{{\mathbf 0}}

\def \bu{{\mathbf u}}
\def \be{{\mathbf e}}

\def \bq{{\mathbf q}}

\usepackage{accents}

\begin{document}

\date{}

\author{Mustafa Aggul \footnotemark[1]\, \hspace{.1in}\hspace{.1in} Song\"{u}l Kaya\footnotemark[2] }
\title{Defect-Deferred Correction Method Based on a Subgrid Artificial Viscosity Model for Fluid-Fluid Interaction}

\maketitle

\renewcommand{\thefootnote}{\fnsymbol{footnote}}

\footnotetext[1]{Department of Mathematics, Hacettepe University, 06800, Ankara, Turkey; email:
mustafaaggul@hacettepe.edu.tr}

\footnotetext[2]{Department of Mathematics, Middle East
Technical University, 06800, Ankara, Turkey; email:
smerdan@metu.edu.tr}


\begin{abstract}
A defect-deferred correction method, increasing both temporal and spatial accuracy, for fluid-fluid interaction problem with nonlinear interface condition is considered by geometric averaging of the previous two-time levels. In the defect step, an artificial viscosity is added only on the fluctuations in the velocity gradient by removing this effect on a coarse mesh. The dissipative influence of the artificial viscosity is further eliminated in the correction step while gaining additional temporal accuracy at the same time. The stability and accuracy analyses of the resulting algorithm are investigated both analytically and numerically.

\end{abstract}


\section{Introduction}

In this paper, we consider a decoupled time stepping method for a fluid-fluid interaction problem. The coupling of Navier-Stokes equations with nonlinear interface condition is typical for terrestrial applications, such as atmosphere-ocean interaction or layers of stratified fluid, see \cite{R1,R2}.  The model is given by:  find (for $i=1,2$) $u_{i}:\Omega_{i}\times [0,T]\rightarrow \mathbb{R}^d$ and $p_{i}: \Omega_{i}\times [0,T]\rightarrow \mathbb{R}$ satisfying (for $0<t\leq T$)
\begin{eqnarray}
\partial_{t}u_i -\nu_i\Delta u_i+u_i \cdot \nabla u_i + \nabla p_i &=&  f_i \qquad \mathrm{in}\ \Omega_i, \label{eq:atmo} \\
-\nu_i \hat{n}_i  \cdot \nabla u_i \cdot \tau &=& \kappa |u_i - u_j|( u_i - u_j )\cdot \tau \quad \mathrm{on}  \ I \ \mathrm{for} \ i,j = 1,2, \ i \neq j \,  , \label{eq:atmoI}\\
u_i \cdot \hat{n}_i  &=&0 \qquad \mathrm{on} \ I \ \mathrm{for} \ i,j = 1,2, \label{eq:atmoNONLINEAR}\\
\nabla \cdot u_i &=&0 \qquad  \mathrm{ in } \ \Omega_i,  \label{eq:atmoNONLINEAR1}\\
u_i(x,0)&=&u_i^0(x) \qquad\mathrm{ in }\ \Omega_i , \label{eq:atmoIC}\\
u_i&=&0 \qquad\mathrm{ on } \ \Gamma_i = \partial\Omega_i\setminus I . \label{eq:atmoBC}
\end{eqnarray}
Here, the domain $\Omega\subset \mathbb{R}^d$, ($d=2,3$) is a polygonal or polyhedral domain that consists of two subdomains $\Omega_1$ and $\Omega_2$, coupled across an interface $I$, for times $t\in [0,T]$. The unknown velocity fields and pressure are denoted by $u_i$ and $p_i$. Also, $|\cdot|$ represents the Euclidean norm and the vectors $\hat{n}_i$ represent the unit normals on $\partial \Omega_i$, and $\tau$ is any vector such that $\tau\cdot\hat{n}_i=0$.  Further, the kinematic viscosity is $\nu_i$ and the body forcing on velocity field  is $f_i$ in each subdomain. Here, $\kappa$ denote the friction parameter for which frictional drag force is assumed to be proportional to the square of the jump of the velocities across the interface.

The main characteristic of the proposed defect-deferred correction (DDC) algorithm is the use of a projection-based variational multiscale method (VMS) as a predictor (defect) step for fluid-fluid interaction problems. Here, the geometric averaging (GA) of the coupling terms is considered at the interface. In VMS, since stabilization acts only on the fluctuations in the velocity gradient, the proposed algorithm is called subgrid artificial viscosity (SAV) based defect-deferred correction (SAV-DDC) method. New SAV based defect step indeed increases the efficiency of the DDC method. The scheme replaces the artificial viscosity (AV) step of the defect-deferred correction (AV-DDC) method of \cite{ACEL18} by the SAV step. For smooth solutions, Theorem 5.3 shows the error of the SAV-DDC algorithm is second order in time. Section 6 includes numerical tests to confirm theory and establishes the advantages of the proposed approach over AV-DDC.

\subsection{Related Works}

In recent years, the atmosphere-ocean interaction problem has been attracted by many scientists to contribute to the simulation of these complex flows. For example, Refs. \cite{B1, B2, LTW1, LTW2} studied modeling of atmosphere-ocean problems and their numerical analysis. Different treatments of coupling terms at the interface are derived to improve the solution of these problems. The method in \cite{BK06} uses nonlinear interface conditions, whereas, in \cite{CHL09}, interface conditions for two heat equations are linearly coupled. In \cite{CHL12}, a decoupling approach, known as GA of the coupling terms at the interface, is introduced for nonlinear coupling of two Navier-Stokes equations. This idea leads to a decoupled and unconditionally stable method. The conditional stability and error estimate of the implicit/explicit (IMEX) method, also proposed in \cite{CHL12}, for a fluid-fluid interaction problem are considered in \cite{ZHS16}. Recently, Aggul et al. derived GA-VMS treatment of fluid-fluid interaction problems in \cite{AEKL20}.

For improving regularity aspects of numerical simulations, one way is to use the defect correction algorithm. Often, the use of the defect correction method leads to over correcting near layers, \cite{H82}. There are various proposals to stabilize the defect correction method, see \cite{EL89, HK88, HK92}. The synthesis of the subgrid stabilization method and defect correction method was studied in \cite{KLR} for steady-state Navier-Stokes equations. This work notes that this combination yields an efficient algorithm and keeps the best algorithmic features. On the other hand, to increase temporal accuracy, one idea is to combine with a deferred correction method along with some stabilizations, see \cite{S78, DG}. We refer to the reader \cite{B84, DG, S78} for comprehensive analytical insight. The predictor-corrector type algorithms considered in \cite{AL17, EL} employ AV for stabilization of all scales. In this approach, the over-diffusive effect of the first step (AV based defect step) is subtracted via a correction step. This discretization idea proved to be the key idea for unconditional stability and higher accuracy without extra computational cost. The defect-deferred correction method with an artificial diffusion step for \eqref{eq:atmo}-\eqref{eq:atmoBC} was studied in \cite{ACEL18}.

As noted in \cite{AL17}, since AV is introduced for all scales, it produces too diffusive approximation. The one considered herein employs SAV stabilization, associated with the VMS method, as a predictor step which introduces AV only for small scales. This novel SAV idea is proposed for convection-dominated equations by Layton \cite{LA01} for the original ideas given in \cite{GU99A,HU00,MAY}. Because of the attractiveness of this approach, there is a substantial amount of works on the projection-based VMS stabilization for fluid problems. We refer the reader to \cite{V17} for review. Also, the same replacement of the predictor step is studied in \cite{A20} for the one-domain Navier Stokes Equation and shown to improve numerical results. This paper is devoted to illustrating how SAV based DDC method contributes to the development of the numerical simulation for fluid-fluid interaction without sacrificing stability and convergence.

\section{Continuous and Discrete Problem}

To write weak formulation, we introduce some notation. As usual, the Sobolev space of functions whose first derivatives are in
$L^2(\Omega_i)$ for each subdomain is denoted by $H^1(\Omega_i), \,(i=1,2)$. The space $H^1_0(\Omega_i)$ denotes the subspace of $H^1(\Omega_i)$ of functions with zero trace on the boundary $\Gamma_i,\,(i=1,2)$ Let $(\cdot,\cdot)_{\Omega_i}$ and $\norm{\cdot}_{\Omega_i}$ denote the $L^2$ inner product and $L^2$ norm, respectively, and $\norm{\cdot}_I$ denotes the $L^3(I)$ norm at the interface.  The Hilbert space $(H^k(\Omega))^d$ is equipped with the norm $\|\cdot\|_k$.  The norm of the dual space of $H^{-1}(\Omega_i)$ of $H_0^1(\Omega_i)$ and the semi-norm of $H^k$, for $1\leq k < \infty$ are denoted by $\|\cdot\|_{-1,\Omega_i}$ and  $|\cdot|_k $, respectively. The other norms are labelled with subscripts.

To pose problem, the following functional spaces are introduced for $i=1,2$.

\begin{eqnarray}
&X_{i}&  :=\{v\in (L^2(\Omega_{i}))^d:\nabla v \in L^2(\Omega_{i})^{d\times d}, \, v= 0\,\mbox{  on  }\,\partial \Omega{_{i} \backslash I, \hspace{0.1in} v\cdot \hat{n}_i = 0 \mbox{  on  } I}\},
\nonumber
\\
&Q_{i}& = L_0^2(\Omega_{i}) := \{q \in L^2(\Omega_{i}): \int_{\Omega_{i}} q \ dx = 0\}. \nonumber
\end{eqnarray}
Here,  $$X=X_1 \times  X_2= \{ \bu=(u_1,u_2); u_i \in X_i, \, i=1,2\}, \quad Q=Q_1 \times  Q_2= \{ \bq=(q_1,q_2); q_i \in Q_i, \, i=1,2\}.$$

For defect-deferred finite element formulation, let $\Pi_H$ be a triangulation of $\Omega_i$ and $\Pi_h$ be a refinement of $\Pi_H$ or $\Pi_h=\Pi_H$. Define finite element spaces by $X_i^h\subset X_i$, $Q_i^h \subset Q_i$, $X^h \subset X, Q^h\subset Q$. Further, for coarse mesh space, define  $L_i^H \subset L^2(\Omega_i)^{d\times d}$. We assume that the spaces satisfy inf-sup condition, see \cite{GR79}.

The weak formulation of \eqref{eq:atmo}-\eqref{eq:atmoBC} can be obtained by multiplying with $(v_{i},q_{i}) \in (X_{i}, Q_{i})$ and integrating over the subdomains $\Omega_i$. Then the problem {reads}: Find $u_i:\Omega_i\times[0,T]\to X_i$, $p_i:\Omega_i\times[0,T]\to Q_i$ satisfying
\begin{eqnarray}
{(\partial_{t}u_i ,v_{i})_{\Omega_i} +\nu_i(\nabla  {u}_{i},\nabla v_{i})_{\Omega_i}
+\kappa \int_{I}(u_i-u_j) |u_i-u_j| v_{i}ds}
\nonumber\\
+(\nabla \cdot u_i,q_{i})_{\Omega_i}+{c_i({u}_{i};{u}_{i},v_{i})}-(p_i,\nabla\cdot v_{i})_{\Omega_i}=(\bff_i,v_{i})_{\Omega_i}, \label{weak1}
\end{eqnarray}
where $c_i(\cdot;\cdot,\cdot)$ denotes the explicitly skew-symmetrized nonlinear form
\begin{equation}
c_i(u;v,w)=\frac{1}{2}(u \cdot \nabla v,w)_{\Omega_i}-\frac{1}{2}(u \cdot \nabla w,v)_{\Omega_i}
\end{equation}
for functions $u,v,w \in X_i$, $i=1,2,\,i\neq j $ on $\Omega_i$.

For SAV based defect deferred correction discretization of  \eqref{eq:atmo}-\eqref{eq:atmoBC} by using the Euler method in time, consider a partition $0=t_0<t_1<\dots<t_{M+1}=T$ of the time interval $[0,T]$ and define $\Delta t=T/(M+1)$, $t_n=n\Delta t$.
\begin{algorithm}[Two Step SAV-DDC method]\label{alg:ddc}

Two step SAV-DDC formulation based on GA applied to the problem \eqref{weak1} reads as follows: {Given $\hat{u}_{h,i}^{0}, \hat{u}_{h,i}^{1} \in X_i^h$, $\hat{p}_{h,i}^{1} \in Q_i^h$,} find $(\hat{u}_{h,i}^{n+1},\hat{p}_{h,i}^{n+1},\mathbb{G}_i^{\mathbb{H},{n+1}})\in (X_{i}^h,Q_{i}^h,L_i^H)$, $n=1,...,M$, satisfying
\\
{\underline {\bf Defect Step}}
\\
\begin{eqnarray}
(\frac{{\hat{u}^{n+1}_{h,i}}-\hat{u}_{h,i}^{n}}{\Delta t} ,v_{h,i})_{\Omega_i} +(\nu_i+{\nu_{T,i}})(\nabla  \hat{u}_{h,i}^{n+1},\nabla v_{h,i})_{\Omega_i} +c_i(\hat{u}_{h,i}^{n+1}; \hat{u}_{h,i}^{n+1},v_{h,i})
-(\hat{p}_{h,i}^{n+1},\nabla \cdot v_{h,i})_{\Omega_i} \nonumber
\\
 + (\nabla\cdot \hat{u}_{h,i}^{n+1},q_{h,i})_{\Omega_i}+\kappa \int_{I}|[\hat{\bu}_h^n]|\hat{u}_{h,i}^{n+1}v_{h,i}ds
-\kappa \int_{I}\hat{u}_{h,j}^n|[{\hat{\bu}}_h^n]|^{1/2}|[{ \hat{\bu}}_h^{n-1}]|^{1/2}v_{h,i}ds \nonumber
\\
=(\bff_i^{n+1},v_{h,i})_{\Omega_i}+{\nu_{T,i}}(\mathbb{G}_i^{\mathbb{H},n},\nabla v_{h,i})_{\Omega_i} \label{BE7}\\
(\mathbb{G}_i^{\mathbb{H},n}-\nabla \hat{u}_{h,i}^n, \mathbb{L}_i^H)_{\Omega_i}=0,\label{alg4}
\end{eqnarray}
for all $(v_{h,i},q_{h,i},\mathbb{L}_i^H) \in (X_{i}^h,Q_{i}^h,L_i^H)$. Here, $\nu_{T,i}$ is user-selected positive, constant parameter and typically $\mathcal{O}(h)$ It might be different in each subdomain $\Omega_i$.

Then, given $\hat{u}_{h,i}^{n+1},\hat{p}_{h,i}^{n+1},\hat{p}_{h,i}^{n}$ and $\tilde{u}_{h,i}^n$, find $(\tilde{u}_{h,i}^{n+1},\tilde{p}_{h,i}^{n+1}) \in (X_{i}^h,Q_i^h)$ satisfying

{\underline{\bf Correction Step}}
\begin{multline}
\left( \frac{\tilde{u}_{h,i}^{n+1} - \tilde{u}_{h,i}^n}{\Delta t} ,v_{h,i} \right)_{\Omega_i}+ (\nu_i+{\nu_{T,i}})\left(\nabla \tilde{u}_{h,i}^{n+1},\nabla v_{h,i} \right)_{\Omega_i}  + c_i\left(\tilde{u}_{h,i}^{n+1}; \tilde{u}_{h,i}^{n+1}, v_{h,i} \right) \\
+
\kappa \int_{I}|[{\bf \tilde{u}}_h^n]|\tilde{u}_{h,i}^{n+1}v_{h,i}ds
-\kappa \int_{I}\tilde{u}_{h,j}^n|[\tilde{\bu}_h^n]|^{1/2}|[{\tilde{\bu}}_h^{n-1}]|^{1/2}v_{h,i}ds \\
-(\tilde{p}_{h,i}^{n+1},\nabla \cdot  v_{h,i})_{\Omega_i}
+ (\nabla\cdot \tilde{u}_{h,i}^{n+1},q_{h,i})_{\Omega_i}
=\left( \frac{f_i^{n+1}+f_i^n}{2},  v_{h,i} \right)_{\Omega_i}
\\
+ \frac{\Delta t(\nu_i+{\nu_{T,i}})}{2}\left( \nabla (\frac {\hat{u}_{h,i}^{n+1}-\hat{u}_{h,i}^n}{\Delta t}),\nabla  v_{h,i}\right)_{\Omega_i}
+ {\nu_{T,i}}\left(\nabla(\frac{\hat{u}_{h,i}^{n+1}+\hat{u}_{h,i}^{n}}{2}) , \nabla v_{h,i}\right)_{\Omega_i}
\\
-\frac {\kappa}{2}\Delta t\int_I \hat{u}_{h,i}^{n+1}(\frac{|[\hat{\bu}_h^{n+1}]|-|[\hat{\bu}_h^n]|}{\Delta t}) v_{h,i}ds
+
\frac {\kappa}{2}\Delta t\int_I |[\hat{\bu}_h^n]|(\frac{\hat{u}_{h,i}^{n+1}-\hat{u}_{h,i}^n}{\Delta t}) v_{h,i}ds
\\
-\kappa \int_I \hat{u}_{h,j}^n|[\hat{\bu}_h^n]|^{1/2}|[\hat{\bu}_h^{n-1}]|^{1/2} v_{h,i}ds
+\frac {\kappa}{2} \int_I |[\hat{\bu}_h^{n+1}]| \hat{u}_{h,j}^{n+1}  v_{h,i} ds+\frac {\kappa}{2} \int_I |[\hat{\bu}_h^n]| \hat{u}_{h,j}^n v_i ds
\\
+\frac{1}{2}c_i(\hat{u}_{h,i}^{n+1}; \hat{u}_{h,i}^{n+1}, v_{h,i})
-\frac{1}{2}c_i(\hat{u}_{h,i}^n; \hat{u}_{h,i}^n, v_{h,i})-\left( \frac{\hat{p}_{h,i}^{n+1}-\hat{p}_{h,i}^n}{2}, \nabla \cdot  v_{h,i} \right)_{\Omega_i},\ \forall  v_{h,i} \in X_{i,h}\label{eq:ddc2}.
\end{multline}
\end{algorithm}
In this method, the first step of the algorithm entails the calculation of the predictor step with an added subgrid eddy viscosity term that is GA-VMS method of \cite{AEKL20}. Since the term $\mathbb{G}_i^{\mathbb{H},n}$ in (\ref{alg4}) is defined on a relatively large scales,  the stabilization is effective only on the small scales with an added term in (\ref{BE7}). In this way, while effective stabilization is less dissipative, the second step of algorithm increases time accuracy from first order to second order without increasing computational cost. In the literature for the choice of the eddy viscosity parameter ${\nu_{T,i}}$ in each subdomian, the Smagorisnky model \cite{s63} or a van Driest damping \cite{P00} seems to be commonly used. Based on ideas from \cite{AEKL20,VS} and the error estimation, numerical studies were performed with ${\nu_{T,i}}=h$.

\section{Preliminaries}
In this section, let us introduce some definitions and inequalities.
\begin{definition} \label{defpro}
The $L^2$ projection ${P}^H $of a given function $\mathbb{L}_i$ onto the finite element space $L_i^H$ is the solution of the following: find $\hat{\mathbb{L}}_i= {P}^H \mathbb{L}_i\in L_i^H$ such that
\begin{eqnarray} \label{pro}
(\mathbb{L}_i-{P}^H \mathbb{L}_i, S_H)=0,
\end{eqnarray}
for all $S_H \in L_i^H$.
\end{definition}
Hence, we get
\begin{eqnarray}\label{pro2}
	\|\mathbb{L}_i-{P}^H \mathbb{L}_i\| \leq CH^k\|\mathbb{L}_i\|_{k+1},
\end{eqnarray}
for all $\mathbb{L}_i \in (L(\Omega_i))^{d\times d}\cap (H^{k+1}(\Omega_i))^{d\times d}$.

One can find in the literature that an efficient implementation of (\ref{alg4}) depends on the choice of space spaces  $L_i^H$, see e.g.,\cite{VS}. Based on our experiences with a projection based variational multiscale method, piecewise constant $P0$ or piecewise linear $P1$ are usually the best choices for $L_i^H$ spaces. For brevity, numerical studies obtained with $L_i^H=Pk$ will be presented here with Taylor-Hood finite elements $P{k+1}-P{k}$. We also note that in our numerical studies, we will use single mesh, which is $H=h$. Although, it is expensive to store the velocity gradient (particularly in $3d$), this choice will be the same as storing 4 additional pressures (3 when symmetric gradient is used in the diffusion term), which has relatively very low degrees of freedom comparing to the degrees of freedom of the velocity space. As we will show later this choice also provides good accuracy.

Denoting the corresponding Galerkin approximations of $(u_i,p_i)$ in $(X_i^h,Q_i^h)$ by $(v_{h,i},q_{h,i})$,  one can assume that the following approximation assumptions (see \cite{GR79}):
\begin{eqnarray}
\inf_{v_{h,i} \in X_i^h} \Big(\|u_i-v_{h,i}\|+\|\nabla(u_i-v_{h,i})\|\Big)&\leq& Ch^{k+1} \|u_i\|_{k+1},\label{inp2}\\\inf_{q_{h,i} \in Q_i^h} \|p_i-q_{h,i}\|&\leq& Ch^{k} \|p_i\|_{k}. \label{inp}
\end{eqnarray}



The following lemmas are required for the analysis and their proof can be found in \cite{CHL12}.
\begin{lemma} \label{Standard_ineqs} Let $\alpha, \beta, \theta \in H^1(\Omega_i)$  for $i=1,2$,  then there exists constants $ C(\Omega_i)>0$ such that
	\begin{eqnarray}
	c_i({\alpha};{\beta},\theta)_{\Omega_i}&\leq& C (\Omega_i) \|\alpha\|^{1/2}_{\Omega_i}\|\nabla \alpha\|_{\Omega_i}^{1/2}\|\nabla \beta\|_{\Omega_i}\|\nabla \theta\|_{\Omega_i},\nonumber\\
	\int_{I}\alpha |[\beta]| \theta &\leq& C(\Omega_i) \|\alpha \|_I ||[\beta]||_I \|\theta \|_I, \nonumber\\
	\|\alpha\|_I &\leq& C(\Omega_i) \left({\color{red} } \|\alpha\|^{1/4}_{\Omega_i}\|\nabla \alpha\|^{3/4}_{\Omega_i} { + \|\alpha\|^{1/6}_{\Omega_i}\|\nabla \alpha\|^{5/6}_{\Omega_i} }\right).
	\end{eqnarray}
\end{lemma}
\begin{proof}
The first two bounds are standard - see, e.g., Lemma 2.1 on p. 1301 of \cite{CHL12}. The third bound can be found in \cite{Galdi94}, see Theorem II.4.1, p. 63.
\end{proof}
\begin{lemma} \label{lem:nnl}Let $\alpha_i \in X_i$, $\theta_j \in X_j$, $\boldsymbol{\beta} \in H^1(\Omega_i)$ and  $\epsilon_i,\epsilon_j,\varepsilon_i,\varepsilon_j$ ($i,j=1,2$) be positive constants, then one
\begin{eqnarray}
\kappa \int_{I} |\alpha_i| |[ \boldsymbol{\beta}]||\theta_j|&\leq& \frac{C\kappa^2}{4}\|\alpha_i\|^2_I ||[ \boldsymbol{\beta}]||_I^2+\frac{\epsilon_j}{\nu_j^5}\|\theta_j\|^2_{\Omega_j}+\frac{\nu_j}{2\epsilon_j}\|\nabla \theta_j\|^2_{\Omega_j},\\
	\kappa \int_{I} |\alpha_i| |[ \boldsymbol{\beta}]||\theta_j|&\leq& C\kappa^6\Big(\frac{\epsilon_i^5}{\nu_i^5}||[ \boldsymbol{\beta}]||_I^6\|\alpha_i\|^2_{\Omega_i}+\frac{\varepsilon_j^5}{\nu_j^5}||[ \boldsymbol{\beta}]||_I^6\|\theta_j\|^2_{\Omega_j}\Big) \nonumber\\&&+\frac{\nu_i}{4\epsilon_i}\|\nabla \alpha_i\|^2+\frac{\nu_j}{4\varepsilon_j} \|\nabla \theta_j\|^2,\\
	\kappa \int_{I} |\alpha_i| |[ \boldsymbol{\beta}]||\theta_j|&\leq&C\kappa^6\|\alpha_i\|^6_I\Big(\frac{\epsilon_1^5}{\nu_1^5}\|\beta_1\|_{\Omega_1}^2+\frac{\epsilon_2^5}{\nu_2^5}\|\beta_2\|^2_{\Omega_2}+\frac{2\varepsilon_j^5}{\nu_j^5}\|\theta_j\|_{\Omega_j}^2\Big)\nonumber\\&&+\frac{\nu_1}{4\epsilon_1}\|\nabla \beta_1\|_{\Omega_1}^2+\frac{\nu_2}{4\epsilon_2}\|\nabla \beta_2\|^2_{\Omega_2}+\frac{\nu_j}{2\beta_j}\|\nabla \theta_j\|_{\Omega_j}^2.
\end{eqnarray}
\end{lemma}
\begin{proof}
{ Use Lemma \ref{Standard_ineqs} and Young's inequality (see Lemma 2.2 on p. 1302 of \cite{CHL12}).}
\end{proof}
Along the paper, we use the following inequality whose proof can be found in \cite{HR86}.
\begin{lemma}\label{gron}[Discrete Gronwall Lemma]
	Let $\gamma_i,\theta_i,\beta_i,\alpha_i$ (for $i\geq 0$), and $\Delta t$, C be a non-negative numbers such that
	\begin{eqnarray*}
		\gamma_M +\Delta t \sum_{i=0}^{M} \theta_i \leq \Delta t \sum_{i=0}^{M} \alpha_i\gamma_i +\Delta t \sum_{i=0}^{M} \beta_i + C,\, \, \forall M\geq 0.
	\end{eqnarray*}
	Assume $\alpha_i\Delta t <1$ for all $i$, then,
	\begin{eqnarray*}
		\gamma_M +\Delta t \sum_{i=0}^{M} \theta_i \leq \exp\Bigg(\Delta t \sum_{i=0}^{M}\theta_i \frac{\alpha_i}{1-\alpha_i\Delta t }\Bigg) \Bigg(\Delta t \sum_{i=0}^{M} \beta_i + C\Bigg),\, \, \forall M\geq 0.
	\end{eqnarray*}
\end{lemma}

\section{Stability and Convergence Analysis of the SAV Based Defect Step}

The stability and convergence analysis of SAV step of Algorithm 2.1. are presented in the current section. Due to the same consideration in \cite{AEKL20},  we only state the corresponding results. The proofs are standard. For stability, letting $v_{h,i}= {\hat{u}^{n+1}_{h,i}}$ in \eqref{BE7} and $\mathbb{L}_i^H =\mathbb{G}_i^{H,n} $ in \eqref{alg4} yields the following theorem.
\begin{theorem}[Stability of the first step approximation]\label{thm:stab_first}
Let $f_i\in L^2(0,T; H^{-1}(\Omega_i))$ for $i=1,2$. {The} scheme \eqref{BE7}-\eqref{alg4} is unconditionally stable and provides the following bound at time step $t=M+1$
	\begin{eqnarray}
	\lefteqn{\|{\hat{u}^{M+1}_{h,1}}\|_{\Omega_1}^2 +\|{\hat{u}^{M+1}_{h,2}}\|_{\Omega_2}^2+\Delta t({\nu_{T,1}} \|\nabla  \hat{u}_{h,1}^{M+1}\|^2_{\Omega_1}+{\nu_{T,2}}\|\nabla  {\hat{u}}_{h,2}^{M+1}\|^2_{\Omega_2})}\nonumber\\
	&&+{\Delta t \sum_{{n=1}}^M \Big(\nu_1\|\nabla  \hat{u}_{h,1}^{n+1}\|^2_{\Omega_1} + {\nu_{T,1}}\|\nabla  \hat{u}_{h,1}^{n+1} - \mathbb{G}_1^{\mathbb{H},n}\|^2_{\Omega_1}+{\nu_{T,1}} \|\nabla  \hat{u}_{h,1}^{n} - \mathbb{G}_1^{\mathbb{H},n}\|^2_{\Omega_1} \Big)}\nonumber\\
	&&+{\Delta t \sum_{{ n=1}}^M \Big(\nu_2\|\nabla  \hat{u}_{h,2}^{n+1}\|^2_{\Omega_2} + {\nu_{T,2}}\|\nabla  \hat{u}_{h,2}^{n+1} - \mathbb{G}_2^{\mathbb{H},n}\|^2_{\Omega_2}+{\nu_{T,2}} \|\nabla  \hat{u}_{h,2}^{n} - \mathbb{G}_2^{\mathbb{H},n}\|^2_{\Omega_2} \Big)}\nonumber\\
	&&+\kappa \Delta t \sum_{{n=1}}^M \int_{I}\Big{|}|[\hat{\bu}_h^n]|^{1/2}\hat{u}_{h,1}^{n+1}-|[\hat{\bu}_h^{n-1}]|^{1/2}\hat{u}_{h,2}^{n}\Big{|}^2ds+\kappa \Delta t \sum_{{n=1}}^M \int_{I}\Big{|}|[\hat{\bu}_h^n]|^{1/2}\hat{u}_{h,2}^{n+1}-|[\hat{\bu}_h^{n-1}]|^{1/2}\hat{u}_{h,1}^{n}\Big{|}^2ds \nonumber\\
		&&+\kappa \Delta t \int_{I}|[\hat{\bu}_h^M]| (|\hat{u}_{h,1}^{M+1}|^2+|\hat{u}_{h,2}^{M+1}|^2)ds\nonumber\\
	&\leq&\|{u^{{1}}_{h,1}}\|_{\Omega_1}^2+\|{u^{{1}}_{h,2}}\|_{\Omega_2}^2+\kappa \Delta t \int_{I}|[\bu_h^{{0}}]| (|u_{h,1}^{{1}}|^2+|u_{h,2}^{{1}}|^2)ds+\Delta t({\nu_{T,1}} \|\nabla u_{h,1}^{{1}}\|_{\Omega_1}^2 +{\nu_{T,2}}\|\nabla u_{h,2}^{{1}}\|_{\Omega_2}^2 )\nonumber\\&&+{\Delta t }\sum_{{ n=1}}^M(\nu_1^{-1}\|f_1^{n+1}\|_{-1,\Omega_1}^2+\nu_2^{-1}\|f_2^{n+1}\|_{-1,\Omega_2}^2)\label{stb}
	\end{eqnarray}
\end{theorem}
\begin{proof}
See Lemma 4.2 in \cite{AEKL20} for the proof.
\end{proof}

In order to establish the accuracy of the first step, we assume that all functions are sufficiently regular, i.e. the solution of \eqref{eq:atmo}-\eqref{eq:atmoBC} satisfies
\begin{eqnarray}
\bu \in{L^\infty(0,T;H^{k+1}(\Omega) \cap H^3(\Omega) )}, \quad
\partial_t \bu \in{L^\infty(0,T;H^{k+1}(\Omega)^d)}\label{reg},\quad \partial_{tt} \bu\in{L^\infty(0,T;H^{1}(\Omega)^d)} \label{r}.
\end{eqnarray}
The following discrete norms are used in the convergence analysis.
\begin{eqnarray}
|||u|||_{\infty,p}=\max\limits_{0\leq j \leq N}||u(t^j)||_p, \, \, |||u|||_{s,p}=\Big(\Delta t \sum\limits_{j=1}^{M}||u(t^j)||_p^s \Big)^{\frac{1}{s}}.
\end{eqnarray}

For finite element error analysis, first the weak formulation \eqref{weak1} is written for $v_{h,i}\in X_h^i $ and evaluated at $t^{n+1}$. Then, the resulting equation is subtracted from \eqref{BE7}-\eqref{alg4}. By using error decomposition, standard a priori error analysis along with tools Lemma \ref{Standard_ineqs}, Lemma \ref{lem:nnl} and Lemma \ref{gron}, one has the following error estimation.

With the notation of \cite{AEKL20}, let
$$D^{n+1} = \tilde{\nu}^5\Big(1+\kappa^6E^{n+1}+|||\nabla u|||_{\infty,\Omega}^4+{(\nu_{T,1}^2+\nu_{T,2}^2)h^{-2}}\Big)$$ where
$\tilde{\nu} = \max\{{(\nu_1+{\nu_{T,1}})^{-1}},{(\nu_2+{\nu_{T,2}})^{-1}}\}$ and $E^{n+1} = \max_{j}\{\max\{\|u(t^j)\|^6_I,\|u_h^j\|^6_I\}\}$ for $j=0,1,...,n+1$.

\begin{theorem}\label{thm:convdefect}{ Let the time step be chosen so that $\Delta t \le 1/D^{n+1}$. Then the following bound on the error holds under the regularity assumptions \eqref{r}:}
	\begin{eqnarray}
	\lefteqn{\|{\bf u}(t^{{M}+1})-{\bf u}^{{M}+1}\|^2  +{ \frac{3}{4}(\nu_1+{\nu_{T,1}})\Delta t\sum_{{n=1}}^{M}\|\nabla  (u_1(t^{n+1})-u_{h,1}^{n+1})\|^2}}\nonumber\\&&+2\kappa\Delta t \sum_{{ n=1}}^{M}\int_{I}|[{\bf u}^n]|  |{\bf u}(t^{n+1})-{\bf u}^{n+1}|^2ds +{ \frac{3}{4}(\nu_2+{\nu_{T,2}})\Delta t\sum_{{ n=1}}^{M}\|\nabla  (u_2(t^{n+1})-u_{h,2}^{n+1})\|^2}
	\nonumber\\&\leq& \|{\bf u}(t^{{1}})-{\bf u}_h^{{1}}\|^2+\frac{(\nu_1{ +{\nu_{T,1}}})\Delta t}{8}(2\|\nabla({ u_1}(t^{{1}})-u_{h,1}^{{1}})\|^2_{\Omega_1}+\|\nabla ({ u_1}(t^{{0}})-u_{h,1}^{{0}})\|^2_{\Omega_1})\nonumber\\[3pt]&&+\frac{(\nu_2{ +{\nu_{T,2}}})\Delta t}{8}(2\|\nabla({ u_{2}}(t^{{1}})-{u_{h,2}^{{1}}})\|^2_{\Omega_2}+\|\nabla ({ u_2}(t^{{0}})-u_{h,2}^{{0}})\|^2_{\Omega_2})\nonumber\\&&+C(\Delta t^2+h^{2k}+{(\nu_{T,2}^2(\nu_2+{\nu_{T,2}})^{-1}+\nu_{T,1}^2(\nu_1+{\nu_{T,1}})^{-1})H^{2k}}),\label{thm}
	\end{eqnarray}
	where $C$ is a generic constant depending only on ${f_i},{\nu_i+{\nu_{T,i}}},\Omega_i$.
\end{theorem}

\begin{remark} \label{rr}
By using approximation inequalities \eqref{inp2}-\eqref{inp}, along with the choices of $H=h$, $\nu_{T,i} = h$ and $L_i^H=P_1$,  the error estimate for SAV based defect step is as follows: for Taylor-Hood finite elements ($P2-P1$)
$$\|error\|\leq C^*(\Delta t+h^2)$$
\end{remark}

\section{Stability and Convergence Analysis of the Correction Step}

It is also necessary to establish the stability of the second step. Both the stability and accuracy of the correction step depend on the accuracy of time derivative of the error in the first step. The following constant and assumptions depending on true solution $\bu$ will be used also in the proofs below.
Let $$C_{\bu}:=  \norm{\nabla \bu (x,t)}_{L^\infty(0,T;L^{\infty}(\Omega))}, \quad C_{\nabla \bu_{i,t}}:=\norm{\nabla \bu_{i,t}(x,t)}_{L^\infty(0,T;L^{\infty}(\Omega))}.$$
{\it Assumption 1}. There exists $\alpha>0$ such that $\alpha \leq \norm{\bu({\bf x},t)}, \forall {\bf x}\in I, \forall t\in(0,T]$.
\\
{\it Assumption 2}. For $i=0,1$, $0 <t \leq \Delta t, \forall  {\bf x} \in I$, true solution $\bu$ satisfies
$$\Big| \frac{\partial}{\partial t} (\bu_i(t))\Big| \leq C(\Delta t)^{1/4}.$$
Most parts of the proof of the following theorem follow the lines of Theorem 8 of \cite{ACEL18}, except subgrid stabilization terms in \eqref{BE7}-\eqref{alg4}. Thus, we give an outline of the proof.

\begin{theorem}[Accuracy of Time Derivative of the Error in Defect Step]\label{thm:convtimederiv}
 Let $u_i(\Delta t) \in H^2(\Omega_i)$,  $\Delta {\bf{u}} \in L^2
(0,T; L^2 (\Omega) )$ and ${\bf{u}}_{tt} ,{\bf{u}}_{t} , {\bf{u}}\in L^2
(0,T; L^2 (\Omega) )$. Let $min(h,\Delta t)< C(\frac{\nu_i+h_i}{\kappa})$. Let also $max(h,\Delta t,\nu_{T,1},\nu_{T,2}) \le \frac{\alpha}{4\sqrt{C^*}}$, where $\alpha$ is the constant introduced in Assumption 1, and $C^*$ is the constant from Remark \ref{rr}.

Then $\exists C >0$ independent of $h, \nu_{T,i}, \,
\Delta t$ such that for any $n \in \{ 0,1,2,\cdots ,
M-1=\frac{T}{\Delta t} - 1 \}$ , the discrete time derivative of the error
satisfies
\begin{multline} \label{timederiverraccuracy}
\lVert \frac{\be^{n+1}-\be^n}{\Delta t} \rVert^2
+(\nu_1+\nu_{T,1})\Delta t \sum_{j=1}^{n}\lVert \nabla \left(\frac{e_1^{j+1}-e_1^j}{\Delta t}\right) \rVert^2
+(\nu_2+\nu_{T,2})\Delta t \sum_{j=1}^{n}\lVert \nabla \left(\frac{e_2^{j+1}-e_2^j}{\Delta t}\right) \rVert^2\\
\leq C\left(h^{2k}+(\Delta t)^2+\nu_{T,1}^2+\nu_{T,2}^2 \right).
\end{multline}
where $e_i^{n} = u_i^{n}-\hat{u}_{h,i}^{n}$ , $i=1,2$.
\end{theorem}
\begin{proof}
In $\Omega_1$, at time level $n+1$, the true solution of \eqref{eq:atmo}-\eqref{eq:atmoBC} satisfies
	\begin{eqnarray}
	\lefteqn{(\frac{{u_1(t^{n+1})}-u_{1}(t^{n})}{\Delta t} ,v_{h,1})_{\Omega_1} +(\nu_1+{\nu_{T,1}})(\nabla  u_{1}(t^{n+1}),\nabla v_{h,1})_{\Omega_1}-(p_1(t^{n+1}),\nabla\cdot v_{h,1})_{\Omega_1}}
	\nonumber\\
	&&+\kappa \int_{I}(u_1(t^{n+1})-u_{2}(t^{n+1}))|[{\bf u}(t^{n+1})]|v_{h,1}ds+ {c_1(u_{1}(t^{n+1});u_{1}(t^{n+1}),v_{h,1})}\nonumber\allowdisplaybreaks\\&=&(\frac{{u_1(t^{n+1})}-u_{1}(t^{n})}{\Delta t}-\partial_{t}u_1(t^{n+1}) ,v_{h,1})_{\Omega_1}+{\nu_{T,1}}(\nabla  u_{1}(t^{n+1}),\nabla v_{h,1})_{\Omega_1}\nonumber\\&&+(\bff_1^{n+1},v_{h,1})_{\Omega_1} \label{err1}
	\end{eqnarray}
	for all $v_{h,1}\in X_1^h$.  By writing down the equation \eqref{BE7} for $i=1$ and subtracting it from \eqref{err1} gives the error equation. Then, for arbitrary $\tilde{u}_1^{n+1}\in X_1^h$ , the error is decomposed into
	\begin{eqnarray}
	e_1^{n+1}&=&u_1(t^{n+1})-\hat{u}_{h,1}^{n+1}=(\tilde{u_1}^{n+1}-\hat{u}_{h,1}^{n+1})-(\tilde{u_1}^{n+1}-u_1^{n+1}) := \phi_1^{n+1}-\eta_1^{n+1}
\nonumber
		\end{eqnarray}
Letting $v_{h,1}=\frac{\phi_1^{n+1}-\phi_1^n}{\Delta t} \in X_{1,h}$ in the resulting equation yields
	
\begin{equation}
\begin{aligned}
\left(\frac {e_1^{n+1}-e_1^n}{\Delta t}, \frac{\phi_1^{n+1}-\phi_1^n}{\Delta t} \right)_{\Omega_1}
+(\nu_1+\nu_{T,1})(\nabla e_1^{n+1},\nabla \frac{\phi_1^{n+1}-\phi_1^n}{\Delta t})_{\Omega_1}
\\
+c_1 (u_1^{n+1} ; u_1^{n+1},\frac{\phi_1^{n+1}-\phi_1^n}{\Delta t})-c_1(\hat{u}_{h,1}^{n+1}; \hat{u}_{h,1}^{n+1},\frac{\phi_1^{n+1}-\phi_1^n}{\Delta t})\\
-(p_1^{n+1}-\hat{p}_{h,1}^{n+1},\nabla \cdot \Big(\frac{\phi_1^{n+1}-\phi_1^n}{\Delta t}\Big))_{\Omega_1}
+\kappa \int_{I}(u_1^{n+1}-u_{2}^{n+1})|[{\bf u}^{n+1}]|\frac{\phi_1^{n+1}-\phi_1^n}{\Delta t}ds
\\
-\kappa \int_{I}|[\hat{\bu}_h^n]|\hat{u}_{h,1}^{n+1}\frac{\phi_1^{n+1}-\phi_1^n}{\Delta t}ds
+\kappa \int_{I}\hat{u}_{h,2}^n|[{\hat{\bu}}_h^n]|^{1/2}|[{ \hat{\bu}}_h^{n-1}]|^{1/2}\frac{\phi_1^{n+1}-\phi_1^n}{\Delta t}ds
\\
=\nu_{T,1}(\nabla u_1^{n+1}-\mathbb{G}_1^{\mathbb{H},n}, \nabla \Big(\frac{\phi_1^{n+1}-\phi_1^n}{\Delta t}\Big))_{\Omega_1}+(\frac{{u_1^{n+1}}-u_{1}^{n}}{\Delta t}-\partial_{t}u_1^{n+1},\frac{\phi_1^{n+1}-\phi_1^n}{\Delta t})_{\Omega_1} \label{es1}
\end{aligned}
\end{equation}
where $u(t^{n+1})=u^{n+1}$.
Then, write down similar error equation at the previous time level, subtract from \eqref{es1} and let $\rho:=\frac{{u_1^{n+1}}-u_{1}^{n}}{\Delta t}-\partial_{t}u_1^{n+1}$. Along with same choice $v_{h,1}=\frac{\phi_1^{n+1}-\phi_1^n}{\Delta t}:= s_1^{h,n+1} \in X_{1}^h$, we obtain
	
	\begin{multline}
\begin{aligned}
\lVert s_1^{n+1} \rVert^2-(s_1^{n+1},s_1^n)+(\nu_1+\nu_{T,1})\Delta t \lVert \nabla s_1^{n+1} \rVert^2 +c_1 (u_1^{n+1} ; u_1^{n+1},s_1^{n+1})-c_1(\hat{u}_{h,1}^{n+1}; \hat{u}_{h,1}^{n+1},s_1^{n+1})\\
+c_1 (u_1^{n} ; u_1^{n},s_1^{n+1})-c_1(\hat{u}_{h,1}^{n}; \hat{u}_{h,1}^{n},s_1^{n+1})
+\Delta t \left(\frac{\hat{p}_{h,1}^{n+1}-\hat{p}_{h,1}^{n}}{\Delta t}  -   \frac{p_1^{n+1}-p_1^n}{\Delta t}  , \nabla \cdot s_1^{n+1}     \right)\\
\\
+ \kappa \int_I u_1^{n+1} |[\bu^{n+1}]| s_1^{n+1} ds- \kappa \int_I u_1^n |[\bu^n]| s_1^{n+1} ds
-\kappa \int_I \hat{u}_{h,1}^{n+1} |[\hat{u}_{h}^n]| s_1^{n+1} ds\\
+\kappa \int_I \hat{u}_{h,1}^n |[\hat{\bu}_{h}^{n-1}]| s_1^{n+1} ds
-\kappa \int_I u_2^{n+1} |[\bu^{n+1}]| s_1^{n+1} ds
+\kappa \int_I u_2^n |[\bu^n]| s_1^{n+1} ds\\
+\kappa \int_I \hat{u}_{h,2}^n |[\hat{\bu}^n]|^{1/2}|[\hat{\bu}^{n-1}]|^{1/2} s_1^{n+1} ds
-\kappa \int_I \hat{u}_{h,2}^{n-1} |[\hat{\bu}_{h,2}^{n-1}]|^{1/2}|[\hat{\bu}_{h,2}^{n-2}]|^{1/2} s_1^{n+1} ds\\
=\nu_{T,1}\Delta t \left( \nabla (\frac{u_1^{n+1}
-u_1^n}{\Delta t}) -\frac{\mathbb{G}_1^{\mathbb{H},n}-\mathbb{G}_1^{\mathbb{H},n-1}}{\Delta t},\nabla s_1^{n+1} \right)
+\Delta t \left( \frac{\rho_1^{n+1}-{\rho_1^n}}{\Delta t},s_1^{n+1} \right)\\
+\Delta t \left(\frac{\eta_1^{n+1}-2 \eta_1^n+\eta_1^{n-1}}{(\Delta t)^2},s_1^{n+1}\right)
+(\nu_1+\nu_{T,1})\Delta t \left(\nabla (\frac{\eta_1^{n+1}
-\eta_1^n}{\Delta t}),\nabla s_1^{n+1}\right). \label{er2}
\end{aligned}
\end{multline}
Stabilization term in the right hand side of \eqref{er2} is treated in the following way.
Applying  Cauchy-Schwarz and Young's inequalities results in

\begin{eqnarray*}
\lefteqn{\nu_{T,1}\Delta t (\nabla \frac{u(t_{n+1}) - u(t_{n})}{\Delta t} - \frac{\mathbb{G}_1^{\mathbb{H},n} - \mathbb{G}_1^{\mathbb{H},n-1}}{\Delta t},\nabla
s^{h,n+1})} \nonumber\\
&\leq& \frac{C\nu_{T,1}^2 \Delta t}{\nu+\nu_{T,1}}\|\nabla \frac{u(t_{n+1}) - u(t_{n})}{\Delta t}\|^2 + \frac{C\nu_{T,1}^2  \Delta t}{\nu+\nu_{T,1}}\|\frac{\mathbb{G}_1^{\mathbb{H},n} - \mathbb{G}_1^{\mathbb{H},n-1}}{\Delta t}\|^2 + \epsilon \Delta t(\nu+\nu_{T,1})\|\nabla s^{h,n+1}\|^2.
\end{eqnarray*}

By the properties of the projection, error decomposition and the inverse inequality, the following result can be found

\begin{eqnarray*} \label{vms1}
\lefteqn{\|\frac{\mathbb{G}_1^{\mathbb{H},n} - \mathbb{G}_1^{\mathbb{H},n-1}}{\Delta t}\|^2 = \|P^H\nabla \frac{u_1^{h,n} - u_1^{h,n-1}}{\Delta t}\|^2} \nonumber\\
&\leq& \|\nabla \frac{u_1^{h,n} - u_1^{h,n-1}}{\Delta t}\|^2 \leq \|\nabla \frac{u(t_{n}) - u(t_{n-1})}{\Delta t}\|^2 + \|\nabla
(\frac{\eta_1^{n} - \eta_1^{n-1}}{\Delta t})\|^2 + \nu_{T,1}^{-2}\| s^{h,n}\|^2.
\end{eqnarray*}
Then, for the rest of the terms, following exactly same lines of the proof of Theorem 8 in \cite{ACEL18} yields the proof of the theorem.	
	\end{proof}
We now present the stability and accuracy of the correction step. The processes of the proofs are similar to the proof of the stability and accuracy results of the correction step in \cite{ACEL18}. By replacing $H_i$ by $\nu_{T,i}$ in the proof of Theorem 9 and Theorem 10 of \cite{ACEL18} one can obtain corresponding results.

\begin{theorem}[Stability of the correction step approximation]\label{thm:stab_second}
Let $\bf{\tilde{u}}^{n+1} \in X^h$ satisfy (\ref{eq:ddc2}) for each $ n \in
\left\{ 0,1,2,\cdots , \frac{T}{\Delta t} - 1 \right\}$ and $e_i^j= u_i^j-\hat{u}_i^j, \,i=1,2$.  Then $\exists \, C>0$
independent of $h, \, \Delta t$ such that ${\bf{\tilde{u}}}^{n+1}$ satisfies:

\begin{multline}
{\left\lVert {{\tilde{u}}_1}^{n+1} \right\rVert}^2 + {\left\lVert {{\tilde{u}}_2}^{n+1} \right\rVert}^2 + (\nu_1+\nu_{T,1}){\Delta t}\sum_{k=1}^{n+1}{\left\lVert \nabla {{\tilde{u}}_1}^{k} \right\rVert}^2 + (\nu_1+\nu_{T,2}){\Delta t}\sum_{k=1}^{n+1}{\left\lVert \nabla {{\tilde{u}}_2}^{k} \right\rVert}^2\\
 + {\kappa} \, {\Delta t} \int_I \left| {\tilde{u}}_1^{n+1}  |[{\bf {\tilde{u}}}^n]|^{1/2}
-{\tilde{u}}_2^n  |[{\bf {\tilde{u}}}^{n-1}]|^{1/2} \right|^2 ds + {\kappa} \, {\Delta t} \int_I \left| {\tilde{u}}_2^{n+1}  |[{\bf {\tilde{u}}}^n]|^{1/2}
-{\tilde{u}}_1^n  |[{\bf {\tilde{u}}}^{n-1}]|^{1/2} \right|^2 ds\\
\leq
\frac{C {\Delta t} }{\nu_1+\nu_{T,1}} \sum_{j=1}^{n}\Biggl[  \lVert \nabla e_1^{j+1} \rVert^2
+ \lVert  e_2^j \rVert  \lVert \nabla e_2^j \rVert  \lVert \nabla e_i^j \rVert^2
\\
+ \lVert \nabla e_i^{j+1} \rVert^2 + \lVert  e_1^{j+1} \rVert  \lVert \nabla e_1^{j+1} \rVert  \lVert \nabla e_i^j \rVert^2 + \lVert  e_1^j \rVert  \lVert \nabla e_1^j \rVert  \lVert \nabla e_i^j \rVert^2+ \lVert \nabla e_2^j \rVert^2  \Biggr]
\\
+\frac{\Delta t}{14(\nu_1+\nu_{T,1})} \sum_{j=1}^n (\lVert \nabla e_2^j \rVert^2
+\lVert  e_2^j \rVert  \lVert \nabla e_2^j \rVert  \lVert \nabla e_i^j \rVert^2    )
\\
+\frac{8\Delta t(\nu_1+\nu_{T,1}}{19} \sum_{j=1}^n \left\{  \Delta t ^2  \lVert \nabla (\frac{e_1^{j+1}-e_1^j}{\Delta t})  \rVert^2+\Delta t^2 C_{\nabla \hat{u}_t}^2)  \right\}
\\
+\frac{8\Delta t(\nu_1+\nu_{T,1})}{19} \sum_{j=1}^n
\Biggl[  (\Delta t)^2   \lVert \nabla \hat{u}_1^{j+1} \rVert^2        \lVert \nabla (\frac{e_1^{j+1}-e_1^j}{\Delta t})  \rVert^2
\\
+(\Delta t)^2     \lVert \nabla \hat{u}_1^{j+1} \rVert^2    C^2_{\nabla \hat{u}_{1,t}}
+(\Delta t)^2  \lVert \nabla \hat{u}_1^j \rVert^2 \lVert \nabla (\frac{e_1^{j+1}-e_1^j}{\Delta t})  \rVert^2
+ (\Delta t)^2   \lVert \nabla \hat{u}_1^j \rVert^2  C^2_{\nabla \hat{u}_{1,t}} \Biggr]
\\
+\frac{19 \Delta t}{(\nu_1+\nu_{T,1})} \sum_{j=1}^n \Biggl[  \nu_{T,1}^2 \lVert \nabla \hat{u}_1^{j+1} \rVert^2+ \lVert  \frac{f_1^{j+1}+f_1^j}{2} \rVert_{-1}^2 \Biggr] +    \frac{\Delta t C_{\nabla u^{n+1}}}{\nu_1+\nu_{T,1}} \sum_{j=1}^n \Biggl[  1
+ \kappa     \lVert \nabla e_i^{j+1} \rVert^2    \Biggr]
\\
+C\Delta t  \sum_{j=1}^n (\lVert  e_1^{j+1} \rVert^{1/2} \lVert \nabla e_1^{j+1} \rVert^{1/2} + \lVert  e_2^{j+1} \rVert^{1/2} \lVert \nabla e_2^{j+1} \rVert^{1/2}                                  )  \lVert \nabla e_i^{j+1} \rVert^2
\end{multline}
\end{theorem}

\begin{theorem}[Accuracy of Correction Step]\label{thm:convcorrectn}
Let the assumptions of Theorems \ref{thm:convdefect} and \ref{thm:convtimederiv} be satisfied. Then
$\exists C >0$ independent of $h, \, \Delta t$ such that for any $n \in \{ 0,1,2,\cdots , M-1=\frac{T}{\Delta t} - 1 \}$, the solution $\tilde{u}_i^{n+1}$ of (\ref{eq:ddc2}) satisfies

\begin{multline}\label{corrconv1}
\lVert {\bf u}^{n+1} - {\bf \tilde{u}}^{n+1} \rVert^2+(\nu+\nu_{T,1})\Delta t \sum_{j=1}^{n+1}\lVert \nabla (u_1^j - \tilde{u}_1^j) \rVert^2 + (\nu+H_2)\Delta t \sum_{j=1}^{n+1}\lVert \nabla (u_2^j - \tilde{u}_2^j) \rVert^2
\\
\leq C\left(h^4 + h^2\Delta t^2 + \nu_{T,1}^4 + \nu_{T,1}^2\Delta t^2 + \nu_{T,1}^4 + \nu_{T,2}^2\Delta t^2 + (\Delta t)^4\right)
\end{multline}
\end{theorem}

\section{Computational Test}

For a verification of the proposed convergence results and a complete accuracy comparison of SAV-DDC and AV-DDC, a manufactured true solution in $\Omega = \Omega_1 \cup \Omega_2$ with $\Omega_1 = [0,1] \times [0,1]$ and $\Omega_2=[0,1] \times [0,-1]$ will be employed:

\begin{equation*}
\begin{aligned}
u_{1,1} &= a\nu_1e^{-2t}x^2(1 - x)^2(1 + y) + ae^{-t}x(1 - x)\nu_1/\sqrt{\kappa a}\\
u_{1,2} &= a\nu_1e^{-2t}xy(2 + y)(1 - x)(2x-1) + ae^{-t}y(2x - 1)\nu_1/\sqrt{\kappa a} \\
u_{2,1} &= a\nu_1 e^{-2t}x^2(1 - x)^2(1 + \frac{\nu_1}{\nu_2}y) \\
u_{2,2} &= a\nu_1 e^{-2t}xy(1 - x)(2x - 1)(2 + \frac{\nu_1}{\nu_2}y).
\end{aligned}
\end{equation*}
Pressures in both domains are taken to be zero everywhere. Also, forcing functions and initial and boundary values are computed so that they comply with the given manufactured true solutions. The interface drag coefficient $\kappa=1$, and the final time $T=1$ are fixed through computations, as we change $a$, $\nu_1$, and $\nu_2$ from one computation to the other; this way computations can be performed for varying Reynolds numbers. Instead of creating a new coarse mesh for the projections of gradients, one can choose a less degree polynomials, see \cite{VS}. Therefore, Taylor-Hood finite elements $(P2/P1)$ for the velocity and pressure pairs are chosen while projections are taken onto $P1$ finite element space. Also discretization parameters, $h$, $\Delta t$ and the eddy viscosity parameter $\nu_T$ have been initially set to the same quantity $1/8$ and halved all together. As a result, a first order of accuracy is expected with the first step approximation while a second order of accuracy should be obtained with the correction step.

For the first $(i=1)$ and the correction $(i=2)$ step approximations, $u_i^h$ and the true solution, $u$ define global errors (total error in two sub-domains as a whole) by:

\vspace*{-0.15cm}
\begin{eqnarray}
||e_i||_{L^2} = ||\,u-u_i^h\,||_{L^2(0,T;L^2(\Omega))}, \nonumber\\
||e_i||_{H^1} = ||\,u-u_i^h\,||_{L^2(0,T;H^1(\Omega))}. \nonumber
\end{eqnarray}

\begin{table}[H]
\caption{Errors and Convergence Rates (CR) with AV-DDC, $\nu_1=0.5$, $\nu_2=0.1$, $a = 1$. \label{table1}}
\begin{tabular}{c|c|c|c|c||c|c|c|c|}
\cline{2-9}
                         & \multicolumn{4}{c||}{First Step} & \multicolumn{4}{c|}{Correction Step} \\ \hline
\multicolumn{1}{|l|}{$1/h$}  & $||e_1||_{L^2}$  & CR & $||e_1||_{H^1}$  & CR & $||e_2||_{L^2}$ & CR & $||e_2||_{H^1}$ & CR \\ \hline
\multicolumn{1}{|l|}{8}  & 3.72374e-03  & -     & 3.05238e-02    & -    & 1.60148e-03   & -    & 1.53617e-02 & -    \\ \hline
\multicolumn{1}{|l|}{16} & 2.38974e-03  & 0.64  & 2.24897e-02    & 0.44 & 7.01101e-04   & 1.19 & 6.26516e-03 & 1.29 \\ \hline
\multicolumn{1}{|l|}{32} & 1.40266e-03  & 0.77  & 1.13031e-02    & 0.99 & 2.52251e-04   & 1.47 & 2.22988e-03 & 1.49 \\ \hline
\multicolumn{1}{|l|}{64} & 7.70160e-04  & 0.86  & 6.23816e-03    & 0.86 & 7.79316e-05   & 1.69 & 6.91693e-04 & 1.69 \\ \hline
\end{tabular}
\end{table}

\begin{table}[H]
\caption{Errors and Convergence Rates (CR) with SAV-DDC, $\nu_1=0.5$, $\nu_2=0.1$, $a = 1$. \label{table2}}
\begin{tabular}{c|c|c|c|c||c|c|c|c|}
\cline{2-9}
                         & \multicolumn{4}{c||}{First Step} & \multicolumn{4}{c|}{Correction Step} \\ \hline
\multicolumn{1}{|l|}{$1/h$}  & $||e_1||_{L^2}$  & CR & $||e_1||_{H^1}$ & CR & $||e_2||_{L^2}$ & CR & $||e_2||_{H^1}$ & CR   \\ \hline
\multicolumn{1}{|l|}{8}  & 1.13217e-03          & -    & 1.20279e-02   & -    & 5.43879e-04   & -    & 8.87426e-03   & -    \\ \hline
\multicolumn{1}{|l|}{16} & 4.01572e-04          & 1.50 & 3.87974e-03   & 1.63 & 1.27978e-04   & 2.09 & 2.25343e-03   & 1.98 \\ \hline
\multicolumn{1}{|l|}{32} & 1.52022e-04          & 1.40 & 1.38604e-03   & 1.48 & 2.88961e-05   & 2.15 & 5.62279e-04   & 2.00 \\ \hline
\multicolumn{1}{|l|}{64} & 6.13662e-05          & 1.31 & 5.65840e-04   & 1.29 & 6.66311e-06   & 2.12 & 1.40459e-04   & 2.00 \\ \hline
\end{tabular}
\end{table}

\begin{table}[H]
\caption{Errors and Convergence Rates (CR) with AV-DDC, $\nu_1=0.005$, $\nu_2=0.001$, $a = 1/\nu_1$. \label{table3}}
\begin{tabular}{c|c|c|c|c||c|c|c|c|}
\cline{2-9}
                         & \multicolumn{4}{c||}{First Step} & \multicolumn{4}{c|}{Correction Step} \\ \hline
\multicolumn{1}{|l|}{$1/h$}  & $||e_1||_{L^2}$  & CR & $||e_1||_{H^1}$  & CR & $||e_2||_{L^2}$ & CR & $||e_2||_{H^1}$ & CR \\ \hline
\multicolumn{1}{|l|}{8}  & 2.35046e-02  & -     & 1.82484e-01    & -    & 1.99361e-02   & -    & 1.58978e-01 & -    \\ \hline
\multicolumn{1}{|l|}{16} & 2.06881e-02  & 0.18  & 1.62522e-01    & 0.17 & 1.46014e-02   & 0.45 & 1.23171e-01 & 0.37 \\ \hline
\multicolumn{1}{|l|}{32} & 1.57385e-02  & 0.39  & 1.29466e-01    & 0.33 & 8.10514e-03   & 0.85 & 7.98903e-02 & 0.62 \\ \hline
\multicolumn{1}{|l|}{64} & 1.04041e-02  & 0.60  & 9.24185e-02    & 0.49 & 3.67593e-03   & 1.14 & 4.67579e-02 & 0.77 \\ \hline
\end{tabular}
\end{table}

\begin{table}[H]
\caption{Errors and Convergence Rates (CR) with SAV-DDC, $\nu_1=0.005$, $\nu_2=0.001$, $a = 1/\nu_1$. \label{table4}}
\begin{tabular}{c|c|c|c|c||c|c|c|c|}
\cline{2-9}
                         & \multicolumn{4}{c||}{First Step} & \multicolumn{4}{c|}{Correction Step} \\ \hline
\multicolumn{1}{|l|}{$1/h$}  & $||e_1||_{L^2}$  & CR & $||e_1||_{H^1}$ & CR & $||e_2||_{L^2}$ & CR & $||e_2||_{H^1}$ & CR   \\ \hline
\multicolumn{1}{|l|}{8}  & 9.40918e-03          & -    & 8.09756e-02   & -    & 7.62025e-03   & -    & 6.96130e-02   & -    \\ \hline
\multicolumn{1}{|l|}{16} & 3.77792e-03          & 1.32 & 3.68978e-02   & 1.13 & 2.57658e-03   & 1.56 & 2.94766e-02   & 1.24 \\ \hline
\multicolumn{1}{|l|}{32} & 1.25924e-03          & 1.59 & 1.43057e-02   & 1.37 & 6.59626e-04   & 1.97 & 1.08124e-02   & 1.45 \\ \hline
\multicolumn{1}{|l|}{64} & 4.15235e-04          & 1.60 & 5.05508e-03   & 1.50 & 1.49754e-04   & 2.14 & 3.52645e-03   & 1.62 \\ \hline
\end{tabular}
\end{table}

Tables \ref{table2}-\ref{table4} clearly illustrate that SAV-DDC achieves the proposed order of accuracies. Even though AV-DDC should likewise provide the same order of accuracy, errors due to this method are significantly high compared to that of SAV-DDC. This can be associated with introducing the AV on all scales, which reduces spatial accuracy. On the other hand, using SAV instead of the first step seems to accelerate convergence and produce better accuracy results.

Next, we present computational results for a qualitative comparison of AV-DDC and SAV-DDC. The problem has been constructed so that a parabolic inflow in the top domain passes a circular (cylindrical) obstacle, see the domains in the Figure \ref{fig:domain}. The lower domain of the problem is a $4\times1$ rectangle. Also, the top domain is a $6\times1$ rectangular channel with a circle of radius $0.05$ centered at $(1, 0.5)$. The horizontal boundary of the top domain extends 1 unit from each side beyond the interface intending to improve the corner compatibility.
\vspace*{-0.25cm}
\begin{figure}[H]
	\includegraphics[width=0.5\linewidth]{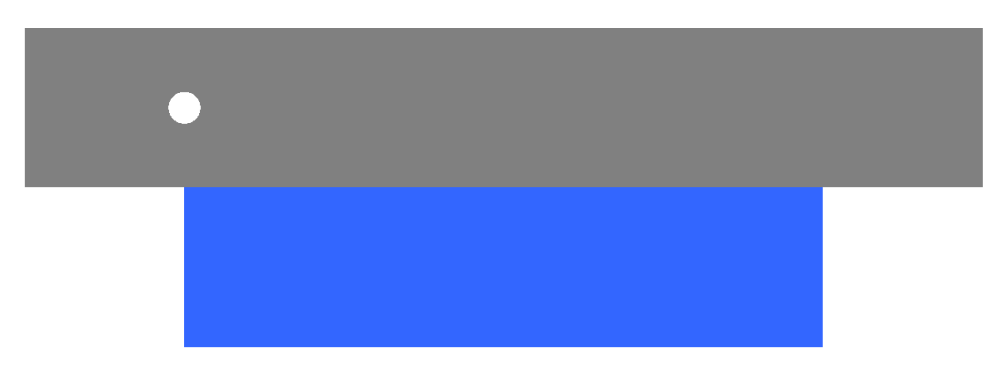}
	\caption{Flow domains}
	\label{fig:domain}
	\vspace*{-0.25cm}
\end{figure}
\vspace*{-0.25cm}
Parabolic flow with an average speed 1 are strongly enforced on all boundaries of the top domain except on the interface (i.e., velocity on the horizontal boundaries except on the interface is zero). Also, the same parabolic inflow is taken to be the initial velocity profile in the whole top domain. No-slip boundary conditions on all the walls of the lower domain have been implemented (except on the interface). The fluid in the lower domain is at rest initially, i.e., there is no flow. Additionally, the problem parameters have been chosen as in Table \ref{table:tableparameters}.

\vspace*{-0.25cm}
\begin{table}[!ht]
	\caption{Problem parameters}
	\begin{center} \label{table:tableparameters}
		\begin{tabular}{|c|c|c|c|c|c|c|c|}
			\hline $\nu_1$  &  $\nu_2$  & $\kappa$ & T   & $\Delta t$ & dofs          & $\nu_{T,1}=\nu_{T,2}$       \\
			\hline   1e-03  &    1      & 1 & 20         & 0.01       & 19748(in $\Omega_1$) - 6511(in $\Omega_2$)  & 0.01                        \\
			\hline
		\end{tabular}
	\end{center}
	\vspace*{-0.5cm}
\end{table}
\vspace*{-0.25cm}

Proceeding without a benchmark problem result, we do not have firm expectations from the flow. On the other hand, we can roughly predict possible outcomes based on our expertise with one domain flows. The setup on the top domain has been extensively used for a qualitative assessment of fluid flows, see, e.g., \cite{AL17,BOWERS20131225,V04}. By choosing $Re=100$, two vortices start to develop behind the cylinder, they then separate into the flow, and then, a vortex street forms. Additionally, the construction on the lower domain is similar to the known lid-driven cavity problem that has a varying velocity pattern on the lid due to the possible eddies above. Consequently, we expect to observe a vortex street formation on the top domain and its influence on the lower domain.

Counter plots of the velocity magnitude with AV-DDC and SAV-DDC are posted in the Figure \ref{fig:obs_comparison}. These plots illustrate that AV-DDC produces too dissipate results to capture vortex street formation on the top domain; it reaches a steady-state after $t=4$. Consequently, we cannot expect completely accurate results in the lower domain. On the other hand, AV-DDC and SAV-DDC seem to produce slightly similar results in the deep lower domain after $t=8$. This can be attributed to the mean flow character of the top domain; upon neglecting eddies, fluid flows horizontally from left to right at a certain speed. However, SAV-DDC seems to respond to the turbulent flow above the lower domain, which is not apparent with AV-DDC. In other words, flow speeds below and above the interface alter mutually, especially after $t=6$.

\begin{figure}[htp]
\begin{center}
\subfigure[AV-DDC at t=2]{\label{av2}\includegraphics[width=0.49\textwidth]{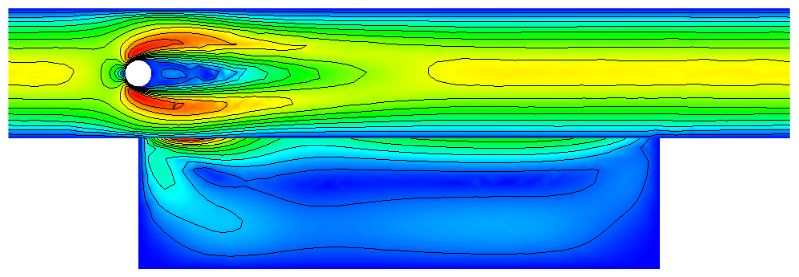}}
\subfigure[SAV-DDC at t=2]{\label{sav2}\includegraphics[width=0.49\textwidth]{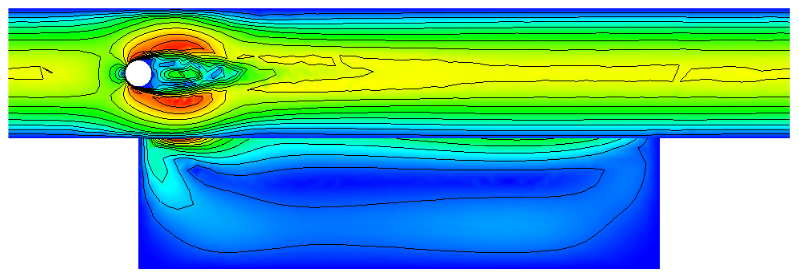}}  \\
\subfigure[AV-DDC at t=4]{\label{av4}\includegraphics[width=0.49\textwidth]{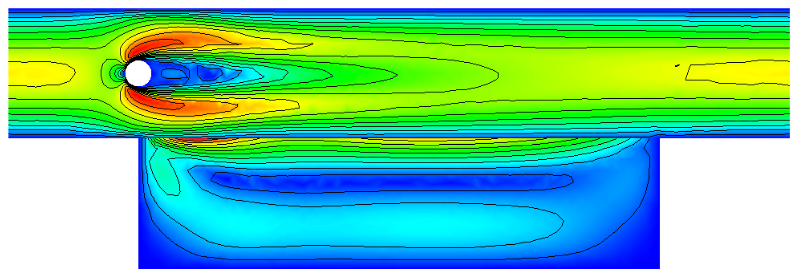}}
\subfigure[SAV-DDC at t=4]{\label{sav4}\includegraphics[width=0.49\textwidth]{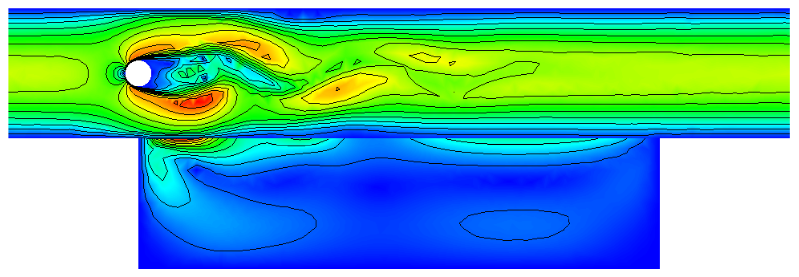}}  \\
\subfigure[AV-DDC at t=6]{\label{av6}\includegraphics[width=0.49\textwidth]{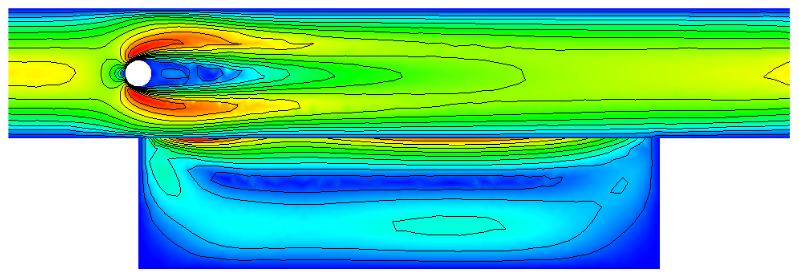}}
\subfigure[SAV-DDC at t=6]{\label{sav6}\includegraphics[width=0.49\textwidth]{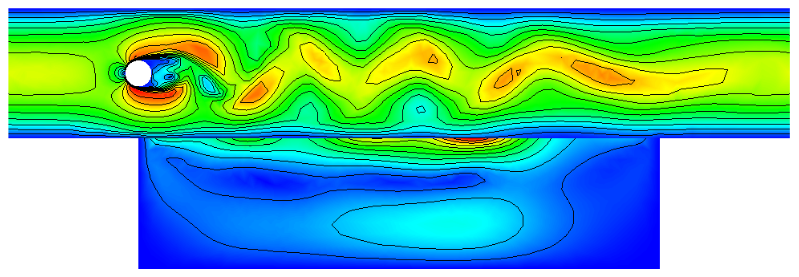}}  \\
\subfigure[AV-DDC at t=8]{\label{av8}\includegraphics[width=0.49\textwidth]{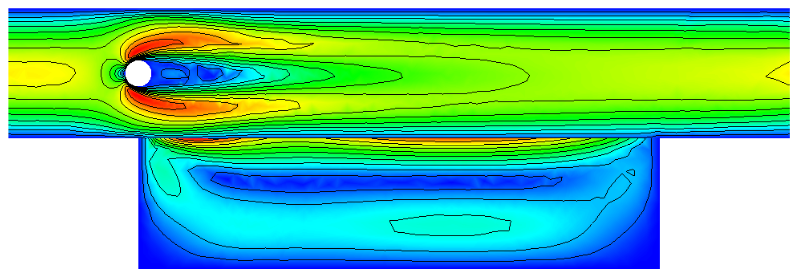}}
\subfigure[SAV-DDC at t=8]{\label{sav8}\includegraphics[width=0.49\textwidth]{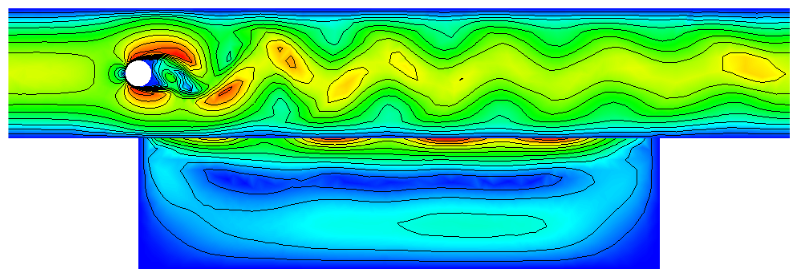}}  \\
\subfigure[AV-DDC at t=10]{\label{av10}\includegraphics[width=0.49\textwidth]{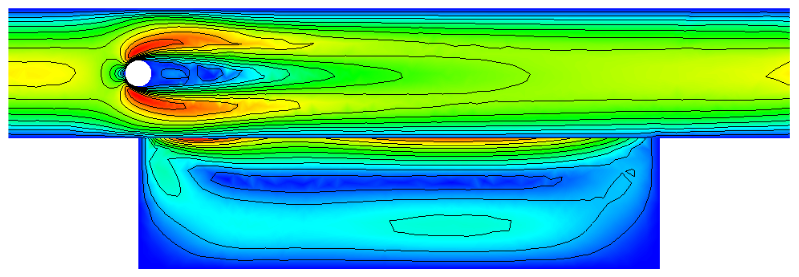}}
\subfigure[SAV-DDC at t=10]{\label{sav20}\includegraphics[width=0.49\textwidth]{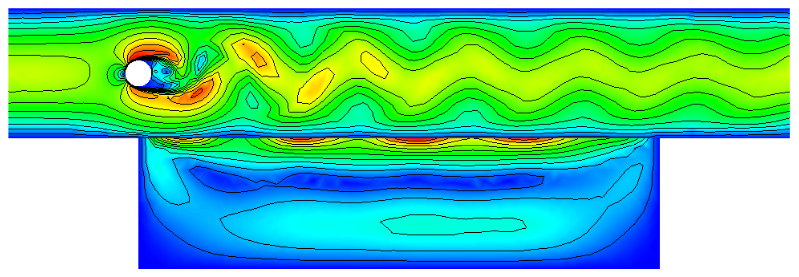}}  \\
\end{center}
\caption{Expected velocity magnitude contours with AV-DDC and SAV-DDC}
\label{fig:obs_comparison}
\vspace*{-0.65cm}
\end{figure}

\section{Conclusion}

In this paper, we develop a defect-deferred correction method based on subgrid artificial viscosity type stabilization in the predictor step for fluid-fluid interaction problems. While defect-deferred correction algorithms with this type of stabilization yields an unconditional stability and second-order convergence rate, with SAV-DDC, efficient and physically more credible approximations are obtained. The numerical examples illustrate the promise of the method when compared with the traditional AV based DDC method.


\end{document}